\documentclass[conference]{IEEEtran1}

\usepackage[utf8]{inputenc}
\usepackage{graphicx}
\usepackage{cite}       
\graphicspath{ {Figures_bnd_wave/} }
\usepackage{amsmath}
\usepackage{amsfonts}
\usepackage{color}
\usepackage{amssymb}
\usepackage{mathrsfs}
\usepackage{float}
\usepackage{enumerate}
\usepackage{verbatim}
\usepackage{setspace}
\usepackage{epsfig}
\usepackage{url}
\usepackage{graphicx}
\usepackage{lscape}

\newcommand{\R}{\mathbb{R}}

\newtheorem{theorem}{Theorem}

\newtheorem{corollary}{Corollary}
\newtheorem{lemma}{Lemma}
\newtheorem{remark}{Remark}
\begin{document}

\title{Finite-dimensional
control of the heat equation: Dirichlet actuation and point measurement}

\author{Rami~Katz and
Emilia~Fridman,~\IEEEmembership{Fellow,~IEEE}
\thanks{R. Katz ({\tt\small rami@benis.co.il}) and E. Fridman ({\tt\small  emilia@eng.tau.ac.il}) are with the School of Electrical Engineering, Tel Aviv University, Israel. }
\thanks{Supported by  Israel Science Foundation (grant no. 673/19) and
by Chana and Heinrich Manderman Chair at Tel Aviv University.}%
}
\markboth{}%
{}

\maketitle

\begin{abstract}
Recently 
finite-dimensional observer-based controllers were introduced for the 1D heat equation, where at least one of the observation or control
operators is bounded. In this paper, for the first time, we manage with such controllers for the 1D heat equation with both operators being unbounded. We consider
Dirichlet actuation and point measurement and use a modal decomposition approach via dynamic extension.
We suggest a direct Lyapunov approach to the full-order closed-loop system, where the finite-dimensional state is coupled with the infinite-dimensional tail of the state Fourier expansion, and provide LMIs for
finding the controller dimension and the resulting exponential decay rate. We further study sampled-data implementation of the controller under sampled-data measurement.
We use Wirtinger-based, discontinuous in time, Lyapunov functionals which compensate sampling in the finite-dimensional state. To compensate sampling in the infinite-dimensional tail, we use a novel form of Halanay's inequality, which is appropriate
for Lyapunov functions with jump discontinuities that do not grow in the jumps. Numerical examples demonstrate the efficiency of the method.
\end{abstract}

\section{Introduction}
Finite-dimensional observer-based control for PDEs is attractive for applications and theoretically challenging.
Such controllers for parabolic systems were designed by the modal decomposition approach in \cite{balas1988finite, christofides2001, curtain1982finite,harkort2011finite}.
The existing results are mostly restricted to bounded control and observation operators, whereas efficient bounds on the observer and controller dimensions are missing.
Thus, the bound suggested in \cite{harkort2011finite} appeared to be highly conservative and difficult to compute.

In our recent paper \cite{RamiContructiveFiniteDim}, the first constructive LMI-based method for finite-dimensional observer-based controller for the 1D heat equation was suggested,
where the controller dimension and the resulting exponential decay rate were found from simple LMI conditions. Robustness of the finite-dimensional controller with respect to input and output delays was studied in \cite{katz2020constructiveDelay}. However, the results
of \cite{RamiContructiveFiniteDim,katz2020constructiveDelay} were confined to cases where at least
one of the observation or control operators is bounded. Sampled-data and delayed boundary control of 1D heat equation under boundary measurement was studied in \cite{katz2020boundary} by using an infinite-dimensional PDE observer. However, finite-dimensional observer-based control of the heat equation in the challenging case where both operators are unbounded remained open. Note that finite-dimensional observer-based control of the 1D linear Kuramoto-Sivashinsky equation (KSE) with both observation and control operators unbounded was studied in \cite{Rami_CDC20}.

In the present paper, for the first time, we manage with finite-dimensional observer-based controllers for the 1D heat equation with both operators unbounded. We consider
Dirichlet actuation and point measurement and employ a modal decomposition approach via dynamic extension.
We suggest a direct Lyapunov approach to the full-order closed-loop system, where the finite-dimensional state is coupled with the infinite-dimensional tail of the state Fourier expansion, and provide LMIs for finding the controller dimension and resulting exponential decay rate. In order to manage with point measurement, we consider $H^1$-stability and apply the Cauchy-Schwarz inequality in a novel form (with fractional powers of the eigenvalues of a Sturm-Liouville operator). Note that for KSE, studied in \cite{Rami_CDC20}, the use of fractional powers of the eigenvalues was not required. We further study sampled-data implementation of the controller under sampled-data measurement, where we consider independent variable samplings of the output and input.
Sampled-data finite-dimensional controllers implemented by zero-order hold devices were suggested in \cite{Aut12,NetzerAut14,kang2018distributed,selivanov2019delayed} for distributed static output-feedback control, \cite{karafyllis2018sampled,karafyllis2017sampledTransport} for boundary state-feeback and in
\cite{katz2020boundary,katz2020constructiveDelay} for observer-based control. Event-triggered sampled-data control of parabolic and hyperbolic PDEs has been studied in \cite{kang2020event,espitia2020event,espitia2020observer}. Due to dynamic extension, in the present paper we suggest sampled-data implementation via a generalized hold device (see e.g. \cite{mirkin2016intermittent} and references therein).
We use Wirtinger-based discontinuous in time Lyapunov functionals which compensate
sampling in the finite-dimensional state and lead to the simplest efficient stability conditions for ODEs \cite{Wir_Aut12,selivanov2016observer}. To compensate sampling in the infinite-dimensional tail, we use a novel form of Halanay's inequality, which is appropriate
for Lyapunov functions with jump discontinuities that do not grow in the jumps. Numerical examples show the efficiency of the proposed method.
\subsection{Notations and mathematical preliminaries}
We denote by $ L^2(0,1)$ the Hilbert space of Lebesgue measurable and square integrable functions $f:[0,1]\to \mathbb{R} $  with the inner product $\left< f,g\right>:=\scriptsize{\int_0^1 f(x)g(x)dx}$ and induced norm $\left\|f \right\|^2:=\left<f,f \right>$.
$H^{k}(0,1)$ is the Sobolev space of  functions  $f:[0,1]\to \mathbb{R} $ having $k$ square integrable weak derivatives, with the norm $\left\|f \right\|^2_{H^k}:=\sum_{j=0}^{k} \left\|f^{(j)} \right\|^2$. The Euclidean norm on $\mathbb{R}^n$ will be denoted by $\left|\cdot \right|$. We denote $f\in H^1_0(0,1)$ if $f\in H^1(0,1)$ and $f(0)=f(1)=0$. For  $P \in \mathbb{R}^{n \times n}$, the notation $P>0$ means that $P$ is symmetric and positive definite. The norm of a matrix $A$ is denoted by $\left|A\right|$. The sub-diagonal elements of a symmetric matrix will be denoted by $*.$ For $U\in \mathbb{R}^{n\times n}, \ U>0$ and $x\in \mathbb{R}^n$ we denote $\left|x\right|^2_U=x^TUx$.

Consider the Sturm-Liouville eigenvalue problem
\begin{equation}\label{eq:SL}
	\begin{aligned}
		\phi''+\lambda \phi = 0,\ \ x\in [0,1]
	\end{aligned}
\end{equation}
with the following boundary conditions:
\begin{equation}\label{eq:2BCs}
	\begin{array}{lll}
		& \phi(0) = \phi(1) = 0.
	\end{array}
\end{equation}
This problem induces a sequence of eigenvalues with corresponding eigenfunctions. The eigenfunctions form a complete orthonormal system in $L^2(0,1)$. The eigenvalues and corresponding eigenfunctions are given by
\begin{equation}\label{eq:SLBCs}
	\begin{array}{lll}
		& \phi_n(x)=\sqrt{2}\sin\left(\sqrt{\lambda_n} x\right),\  \lambda_n = n^2\pi^2, \ n\geq 1.
	\end{array}
\end{equation}

The following lemma will be used:
\begin{lemma}[\cite{RamiContructiveFiniteDim}]\label{lem:H1Equiv}
Let $h\in L^2(0,1)$ satisfy $h \overset{L^2(0,1)}{=} \sum_{n=1}^{\infty}h_n\phi_n$. Then $h\in H^1_0(0,1)$ if and only if $\sum_{n=1}^{\infty}\lambda_nh_n^2< \infty$. Moreover,
\begin{equation}\label{lem22}
\left\|h' \right\|^2=\sum_{n=1}^{\infty}\lambda_nh_n^2.
\end{equation}
\end{lemma}
\section{Continuous-time control of a heat equation}
In this section we consider stabilization of the linear 1D heat equation
\begin{equation}\label{eq:LinearHeat}
	\begin{array}{lll}
		&z_t(x,t) = z_{xx}(x,t) +az(x,t), \ t\geq 0
	\end{array}
\end{equation}
where $x\in [0,1]$, $z(x,t)\in \mathbb{R}$ and $a\in \mathbb{R}$ is the reaction coefficient. We consider Dirichlet actuation given by
\begin{equation}\label{eq:BCsHeatDir}
	 z(0,t)=u(t), \quad z(1,t) = 0
\end{equation}
where $u(t)$ is a control input to be designed and in-domain point measurement given by
\begin{equation}\label{eq:InDomPointMeasHeat}
y(t) = z(x_*,t), \ x_*\in (0,1).
\end{equation}
Following \cite{prieur2018feedback}, we introduce the change of variables
\begin{equation}\label{eq:ChangeVarsHeatDir}
w(x,t)=z(x,t)-r(x)u(t), \quad r(x):=1-x
\end{equation}
to obtain the following equivalent ODE-PDE system
\begin{equation}\label{eq:PDE1HeatDir}
\begin{aligned}
& w_t(x,t)=w_{xx}(x,t)+aw(x,t)+ar(x)u(t)-r(x)v(t),\\
& \dot{u}(t)=v(t), \quad t\geq 0
\end{aligned}
\end{equation}
with boundary conditions
\begin{equation}\label{eq:PDE1HeatDirBCs}
\begin{array}{lll}
&w(0,t)=0, \quad w(1,t)=0.
\end{array}
\end{equation}
and measurement
\begin{equation}\label{eq:InDomPointMeas1HeatDir}
y(t) = w(x_*,t)+r(x_*)u(t).
\end{equation}
Henceforth we will treat $u(t)$ as an additional state variable and $v(t)$ as the control input. Given $v(t)$, $u(t)$ can be computed by integrating $\dot{u}(t)=v(t)$, where we choose $u(0)=0$. Note that this choice implies $z(\cdot,0)=w(\cdot,0)$.

We present the solution to \eqref{eq:PDE1HeatDir} as
\begin{equation}\label{eq:WseriesHeatDir}
\begin{array}{lll}
w(x,t) &= \sum_{n=1}^{\infty}w_n(t)\phi_n(x), \ w_n(t) =\left<w(\cdot,t),\phi_n\right>,
\end{array}
\end{equation}
with $\phi_n(x), \ n\geq 1$ defined in \eqref{eq:SLBCs}. By differentiating under the integral sign, integrating by parts and using \eqref{eq:SL} and \eqref{eq:2BCs} we obtain
\begin{equation}\label{eq:WOdesHeatDir}
\begin{array}{lll}
&\dot{w}_n(t) = (-\lambda_n+a)w_n(t) + ab_nu(t) -b_nv(t),\quad t\geq 0  \\
& b_n=\left<r,\phi_n\right>= \sqrt{\frac{2}{\lambda_n}}, \ w_n(0) = \left<w(\cdot,0),\phi_n\right>, \ n\geq 1.
\end{array}
\end{equation}
In particular note that
\begin{equation}\label{eq:AssbnNonDelayedHeatDir}
\begin{array}{lll}
&b_n \neq 0 , \quad n\geq 1
\end{array}
\end{equation}
and
\begin{equation}\label{eq:AssbnNonDelayed1HeatDir}
\sum_{n=N+1}^{\infty}b_n^2 \leq \frac{2}{\pi^2}\int_N^{\infty}\frac{dx}{x^2}=\frac{2}{\pi^2N}, \ N\geq 1.
\end{equation}
\begin{remark}
State-feedback boundary control of $1$D parabolic PDEs, without dynamic extension, has been suggested in \cite{karafyllis2018sampled}. Without dynamic extension, modal decomposition of \eqref{eq:LinearHeat} with boundary conditions \eqref{eq:BCsHeatDir} results in ODEs similar to \eqref{eq:WOdesHeatDir}, without $v(t)$, where $|b_n|\approx \lambda_n^{\frac{3}{2}}$. The growth of $\left\{b_n\right\}_{n=1}^{\infty}$ poses a problem in compensating cross terms which arise in the Lyapunov stability analysis (see \eqref{eq:WCrosTermNonDelayedHeatDir} below). As can be seen in \eqref{eq:AssbnNonDelayed1HeatDir}, the use of dynamic extension leads to $\left\{b_n\right\}_{n=1}^{\infty}\in l^2(\mathbb{N})$.
\end{remark}

Let $\delta>0$ be a desired decay rate and let $N_0 \in \mathbb{N}$ satisfy
\begin{equation}\label{eq:N0HeatDir}
-\lambda_n+a<-\delta, \quad n>N_0.
\end{equation}
Let $N\in \mathbb{N}, \ N_0\leq N$. $N_0$ will define the dimension of the controller and $N$ will define the dimension of the observer.

We construct a finite-dimensional observer of the form
\begin{equation}\label{eq:WhatSeriesHeatDir}
\hat{w}(x,t): = \sum_{n=1}^{N}\hat{w}_n(t)\phi_n(x)
\end{equation}
where $\hat{w}_n(t)$ satisfy the ODEs for $t\geq 0$:
\begin{equation}\label{eq:WobsODENonDelayedHeatDir1}
\begin{array}{lll}
&\dot{\hat{w}}_n(t) = (-\lambda_n+a)\hat{w}_n(t) +ab_nu(t)- b_nv(t)\\
&\hspace{13mm}-l_n\left[\hat{w}(x_*,t)+r(x_*)u(t)- y(t)\right],\ n\geq 1,\\
&\hat{w}_n(0)=0, \quad 1\leq n\leq N.
\end{array}
\end{equation}
with $y(t)$ in \eqref{eq:InDomPointMeas1HeatDir} and saclar observer gains $l_n,\ 1\leq n\leq N$.

\textbf{Assumption 1:} The point $x_*\in (0,1)$ satisfies
\begin{equation}\label{eq:AsscnNonDelayed}
c_n=\phi_n(x_*)=\sqrt{2}\sin\left(\sqrt{\lambda_n} x_*\right)\neq 0 , \ 1\leq n \leq N_0.
\end{equation}
Note that this assumption is satisfied if $x_*$ is irrational number. 
In this case $\phi_n(x_*)\neq 0$ for all $n\in \mathbb{N}$.\\[0.1cm]
Let
\begin{equation}\label{eq:C0A0HeatDir}
\begin{array}{lll}
&A_0 = \operatorname{diag}\left\{-\lambda_1+a,\dots,-\lambda_{N_0}+a \right\},\\
&B_0 = \left[b_1,\dots,b_{N_0} \right],\ L_0 = \left[l_1,\dots,l_{N_0} \right]^T,\\
&C_0=\left[c_1,\dots,c_{N_0} \right],\ \tilde{B}_0= \left[1,-b_1,\dots,-b_{N_0} \right],\\
&\tilde{A}_0 =\begin{bmatrix}0& 0\\ aB_0&A_0 \end{bmatrix}\in \mathbb{R}^{(N_0+1)\times(N_0+1)}.\\
\end{array}
\end{equation}
Under Assumption 1 it can be verified that the pair $(A_0,C_0)$ is observable by the Hautus lemma. We choose $L_0 = \left[l_1,\dots,l_{N_0} \right]^T\in \mathbb{R}^{N_0}$ which
satisfies the Lyapunov inequality
\begin{equation}\label{eq:GainsDesignLHeatDir}
P_{\text{o}}(A_0-L_0C_0)+(A_0-L_0C_0)^TP_{\text{o}} < -2\delta P_{\text{o}},
\end{equation}
with $0<P_{\text{o}}\in \mathbb{R}^{N_0\times N_0}$. We choose $l_n=0, \ n>N_0$.\\[0.2cm]
Since $b_n \neq 0 , \ n\geq 1$ the pair $(\tilde{A}_0,\tilde{B}_0)$ is controllable. Let $K_0\in \mathbb{R}^{1\times (N_0+1)}$ satisfy
\begin{equation}\label{eq:GainsDesignKHeatDir}
\begin{aligned}
&P_{\text{c}}(\tilde{A}_0+\tilde{B}_0K_0)+(\tilde{A}_0+\tilde{B}_0K_0)^TP_{\text{c}} < -2\delta P_{\text{c}},
\end{aligned}
\end{equation}
with $0<P_{\text{c}}\in \mathbb{R}^{(N_0+1)\times (N_0+1)}$. We propose a $(N_0+1)$-dimensional controller of the form
\begin{equation}\label{eq:WContDefHeatDir}
\begin{aligned}
& v(t)= K_0\hat{w}^{N_0}(t),\\
& \hat{w}^{N_0}(t) = \left[u(t),\hat{w}_1(t),\dots,\hat{w}_{N_0}(t) \right]^T
\end{aligned}
\end{equation}
which is based on the $N$-dimensional observer \eqref{eq:WhatSeriesHeatDir}. \\[0.2cm]
For well-posedness of the closed-loop system \eqref{eq:PDE1HeatDir} and \eqref{eq:WobsODENonDelayedHeatDir1} subject to the control input \eqref{eq:WContDefHeatDir} we consider the operator
\begin{equation}\label{eq:CalADef}
\begin{array}{lll}
&\mathcal{A}_1:\mathcal{D}(\mathcal{A}_1)\subseteq L^2(0,1)\to L^2(0,1), \ \ \mathcal{A}_1w = -w_{xx},\vspace{0.1cm}\\
&\mathcal{D}(\mathcal{A}_1) = \left\{w\in H^2(0,1)| w(0)=w(1)=0\right\}.
\end{array}
\end{equation}
Since $\mathcal{A}_1$ is positive, it has a unique positive square root with domain
\begin{equation}\label{eq:CalARootDef}
\mathcal{D}\left(\mathcal{A}_1^{\frac{1}{2}}\right) \overset{\eqref{lem22}}{=} H^1_0(0,1).
\end{equation}
Let $\mathcal{H}=L^2(0,1)\times \mathbb{R}^{N+1}$ be a Hilbert space with the norm $\left\|\cdot\right\|_{\mathcal{H}}=\sqrt{\left\|\cdot\right\|+\left|\cdot\right|}$.
Defining the state $\xi(t)$ as
\begin{equation*}
\begin{array}{lll}
&\xi(t)=\text{col}\left\{w(\cdot,t),\hat{w}^N(t)\right\}, \vspace{0.1cm}\\ &\hat{w}^{N}(t)=\text{col}\left\{u(t),\hat{w}_1(t),\dots,\hat{w}_N(t)\right\}
\end{array}
\end{equation*}
by arguments of \cite{RamiContructiveFiniteDim},  it can be shown that the closed-loop system \eqref{eq:PDE1HeatDir} and \eqref{eq:WobsODENonDelayedHeatDir1} with control input \eqref{eq:WContDefHeatDir} and initial condition $w(\cdot,0)\in \mathcal{D}\left(\mathcal{A}_1^{\frac{1}{2}}\right)$ has a unique classical solution
\begin{equation}\label{eq:Classical1}
\xi \in C\left([0,\infty);\mathcal{H}\right)\cap C^1\left((0,\infty);\mathcal{H}\right)
\end{equation}
such that
\begin{equation}\label{eq:Classical2}
\xi(t) \in \mathcal{D}\left(\mathcal{A}_1\right)\times \mathbb{R}^{N+1}, \quad t>0.
\end{equation}
Let $e_n(t)$ be the estimation error defined by
\begin{equation}\label{eq:WEstErrorNonDelayed}
e_n(t) = w_n(t)-\hat{w}_n(t), \ 1\leq n \leq N.
\end{equation}
By using \eqref{eq:InDomPointMeas1HeatDir}, \eqref{eq:WseriesHeatDir} and \eqref{eq:WhatSeriesHeatDir}, the last term on the right-hand side of \eqref{eq:WobsODENonDelayedHeatDir1} can be written as
\begin{equation}\label{eq:WIntroZetaNonDelayedHeatDir}
\begin{array}{ll}
&\hat{w}(x_*,t)+r(x_*)u(t) - y(t) =-\sum_{n=1}^{N} c_ne_n(t)-\zeta(t),
\end{array}
\end{equation}
where
\begin{equation}\label{eq:zetaintegralHeatDir}
\begin{array}{ll}
&\zeta(t) = w(x_*,t)-\sum_{n=1}^N w_n(t)\phi_n(x_*)\\
&\overset{\eqref{eq:2BCs},\eqref{eq:PDE1HeatDirBCs}}{=}\int_0^{x_*}\left[w_x(x,t)-\sum_{n=1}^Nw_n(t)\phi_n'(x) \right]dx.
\end{array}
\end{equation}
Then the error equations have the form
\begin{equation}\label{eq:WenHeatDir}
\begin{array}{ll}
&\dot e_n(t)=(-\lambda_n+a)e_n(t)\\
&\hspace{8mm}-l_n\left(\sum_{n=1}^{N} c_ne_n(t)+\zeta(t)\right), \quad t\geq 0.
\end{array}
\end{equation}
Note that $\zeta(t)$ satisfies the following estimate:
\begin{equation}\label{eq:zetaestHeatDir}
\begin{array}{lll}
&\zeta^2(t)\overset{\eqref{eq:zetaintegralHeatDir}}{\leq} \left\|w_x(\cdot,t)-\sum_{n=1}^Nw_n(t)\phi_n'(\cdot) \right\|^2 \\
&\hspace{9mm}\overset{\eqref{lem22}}{\leq} \sum_{n=N+1}^{\infty}\lambda_nw_n^2(t).
\end{array}
\end{equation}
Let
\begin{equation}\label{eq:ErrDefNonDelayedHeatDir}
\begin{array}{lllllll}
&e^{N_0}(t)=\left[e_1(t),\dots,e_{N_0}(t) \right],\\
&e^{N-N_0}(t)=\left[e_{N_0+1}(t),\dots,e_{N}(t) \right]^T,\\
&\hat{w}^{N-N_0}(t)=\left[\hat{w}_{N_0+1}(t),\dots,\hat{w}_{N}(t) \right]^T,\\
& X(t) = \text{col}\left\{\hat{w}^{N_0}(t),e^{N_0}(t),\hat{w}^{N-N_0}(t), e^{N-N_0}(t) \right\},
\end{array}
\end{equation}
and
\begin{equation}\label{eq:ErrDefNonDelayedHeatDir1}
\begin{array}{lllllll}
& A_1 = \operatorname{diag}\left\{-\lambda_{N_0+1}+a,\dots,-\lambda_{N}+a \right\},\\
&B_1 =\left[b_{N_0+1},\dots,b_N \right]^T,\ C_1 =\left[c_{N_0+1},\dots, c_N \right],\\
&\mathrm{a}=\begin{bmatrix}-a,0_{1\times N_0} \end{bmatrix},\ \tilde{L}_0 = \text{col}\left\{0_{1\times 1},L_0 \right\}\\
& \tilde{K}_0 = \begin{bmatrix} K_0+\mathrm{a},&0_{1\times (2N-N_0)}\end{bmatrix},\\
&\mathcal{L}=\text{col}\left\{\tilde{L}_0,-L_0,0_{2(N-N_0)\times 1}\right\},\\
& F = \scriptsize\begin{bmatrix}\tilde{A}_0+\tilde{B}_0K_0 & \tilde{L}_0C_0 & 0 &\tilde{L}_0C_1 \\ 0 & A_0-L_0C_0 & 0 & -L_0C_1\\ -B_1\left(K_0+\mathrm{a}\right) & 0 & A_1 & 0\\
0 & 0 & 0 & A_1 \end{bmatrix}.
\end{array}
\end{equation}
From \eqref{eq:WOdesHeatDir}, \eqref{eq:WobsODENonDelayedHeatDir1},  \eqref{eq:WContDefHeatDir} and \eqref{eq:ErrDefNonDelayedHeatDir1} we have the closed-loop system for $t\geq 0$:
\begin{equation}\label{eq:ClosedLoopHeatDir1}
\begin{array}{lll}
&\dot{X}(t) = FX(t)+\mathcal{L}\zeta(t),\\
& \dot{w}_n(t) = (-\lambda_n+a)w_n(t) -b_n\tilde{K}_0X(t), \ n>N.
\end{array}
\end{equation}
For $H^1$-stability analysis of the closed-loop system
\eqref{eq:ClosedLoopHeatDir1} we define the Lyapunov function
\begin{equation}\label{eq:VNonDelayedHeatDir1}
V(t)=\left|X(t)\right|^2_P+\sum_{n=N+1}^{\infty}\lambda_n w_n^2(t),
\end{equation}
where $P\in \mathbb{R}^{(2N+1)\times (2N+1)}$ satisfies $P>0$. This function is chosen to compensate $\zeta(t)$ using the estimate \eqref{eq:zetaestHeatDir}. Differentiating $V(t)$ along the solution to \eqref{eq:ClosedLoopHeatDir1} gives
\begin{equation}\label{eq:WStabAnalysisNonDelayedHeatDir1}
\begin{array}{lll}
&\hspace{-3mm}\dot{V}+2\delta V = X^T(t)\left[PF +F^TP+2\delta P\right]X(t)\\
&\hspace{-3mm}-2X^T(t)P\mathcal{L}\zeta(t)+2\sum_{n=N+1}^{\infty}\left(-\lambda_n^2+(a+\delta)\lambda_n\right)w_n^2(t)\\
 &\hspace{-3mm}-2\sum_{n=N+1}^{\infty}\lambda_n w_n(t)b_n\tilde{K}_0X(t), \quad t\geq 0.
\end{array}
\end{equation}
Note that since $\lambda_n=n^2\pi^2$, similar to \eqref{eq:WOdesHeatDir} we have
 \begin{equation}\label{3over2}
\begin{array}{lll}
\sum_{n=N+1}^{\infty}\lambda_n^{-\frac{3}{4}}
\le  \pi^{-{3\over 2}} \int_N^{\infty}x^{-\frac{3}{2}}dx=\frac{2}{\sqrt{N}\pi^{3\over 2}}.
\end{array}
\end{equation}
Since $b_n=\sqrt{2\over \lambda_n}$, the Cauchy-Schwarz inequality  implies
\begin{equation}\label{eq:WCrosTermNonDelayedHeatDir}
\begin{array}{lllll}
&-2\sum_{n=N+1}^{\infty}\lambda_n w_n(t)b_n\tilde{K}_0X(t)\\
&\leq 2\sum_{n=N+1}^{\infty}\left[\lambda_n^{\frac{7}{8}}\left|w_n(t)\right|\right]\left[\sqrt{2}\lambda_n^{-\frac{3}{8}}\left|\tilde{K}_0X(t)\right|\right]\\
&\leq \frac{1}{\alpha_1} \sum_{n=N+1}^{\infty}\lambda_n^{\frac{7}{4}}w_n^2(t)\\
&+ 2\alpha_1 \left(\sum_{n=N+1}^{\infty}\lambda_n^{-\frac{3}{4}}\right)\left|\tilde{K}_0X(t)\right|^2\\
&\overset{\eqref{3over2}}{\leq} \frac{1}{\alpha_1} \sum_{n=N+1}^{\infty}\lambda_n^{\frac{7}{4}} w_n^2(t)+ \frac{4\alpha_1 }{\sqrt{N}\pi^{\frac{3}{2}}}\left|\tilde{K}_0X(t)\right|^2
\end{array}
\end{equation}
where $\alpha>0$. From monotonicity of $\lambda_n$ we have
\begin{equation}\label{eq:zetaintroducLMIs}
\begin{array}{lll}
&2\sum_{n=N+1}^{\infty}\left(-\lambda_n^2+(a+\delta)\lambda_n\right)w_n^2(t)\\ &+2\sum_{n=N+1}^{\infty}\lambda_n w_n(t)(-b_n)\tilde{K}_0X(t)\\
&\overset{\eqref{eq:WCrosTermNonDelayedHeatDir}}{\leq}2\sum_{n=N+1}^{\infty}\left(-\lambda_n^2+\frac{1}{2\alpha_1}\lambda_n^{\frac{7}{4}}+(a+\delta) \lambda_n \right)w_n^2(t)\\
&\quad + \frac{4\alpha_1 }{\sqrt{N}\pi^{\frac{3}{2}}}\left|\tilde{K}_0X(t)\right|^2\\
&\leq -2\left(\lambda_{N+1}-a-\delta-\frac{1}{2\alpha_1}\lambda_{N+1}^{\frac{3}{4}} \right)\sum_{n=N+1}^{\infty}\lambda_nw_n^2(t)\\
&\quad + \frac{4\alpha_1 }{\sqrt{N}\pi^{\frac{3}{2}}}\left|\tilde{K}_0X(t)\right|^2\\
&\overset{\eqref{eq:zetaestHeatDir}}{\leq}\! -2\!\left(\lambda_{N+1}\!-a\!-\!\delta-\frac{1}{2\alpha_1}\lambda_{N+1}^{\frac{3}{4}} \right)\zeta^2(t)+\! \frac{4\alpha_1 }{\sqrt{N}\pi^{\frac{3}{2}}}\left|\tilde{K}_0X(t)\right|^2
\end{array}
\end{equation}
provided $\lambda_{N+1}-a-\delta-\frac{1}{2\alpha_1}\lambda_{N+1}^{\frac{3}{4}}\geq 0$.\\[0.1cm]
Let $\eta(t) = \text{col}\left\{X(t),\zeta(t) \right\}$. From \eqref{eq:WStabAnalysisNonDelayedHeatDir1}, \eqref{eq:WCrosTermNonDelayedHeatDir} and \eqref{eq:zetaintroducLMIs} we obtain
\begin{equation}\label{eq:WStabResultNonDelayedHeatDir}
\begin{array}{ll}
&\dot{V}+2\delta V\leq \eta^T(t)\Psi^{(1)} \eta(t)\leq 0, \quad t\geq 0
\end{array}
\end{equation}
if
\begin{equation}\label{eq:WStabResultNonDelayedHeatDir1}
\begin{array}{lll}
&\Psi^{(1)} = \begin{bmatrix}
\Phi^1 & P\mathcal{L}\\
* & -2(\lambda_{N+1}-a-\delta)+\frac{1}{\alpha_1}\lambda_{N+1}^{\frac{3}{4}}
\end{bmatrix}<0,\\
&\Phi^{(1)} = PF+F^TP+2\delta P +\frac{4\alpha_1}{\sqrt{N}\pi^{\frac{3}{2}}}\tilde{K}_0^T\tilde{K}_0.
\end{array}
\end{equation}
By Schur complement \eqref{eq:WStabResultNonDelayedHeatDir1} holds if and only if
\begin{equation}\label{eq:WStabResultNonDelayedHeatDir2}
\begin{bmatrix}
\Phi^{(1)} & P\mathcal{L} & 0\\
* & -2(\lambda_{N+1}-a-\delta) & 1\\
* & * & -\alpha_1\lambda_{N+1}^{-\frac{3}{4}}
\end{bmatrix}<0.
\end{equation}
Note that the LMI \eqref{eq:WStabResultNonDelayedHeatDir2} has \emph{$N$-dependent}  coefficients and its dimension depends on $N$. Summarizing, we arrive at:
\begin{theorem}\label{Thm:WdynExtensionHeatDir}
Consider \eqref{eq:PDE1HeatDir} with boundary conditions \eqref{eq:PDE1HeatDirBCs}, in-domain point measurement \eqref{eq:InDomPointMeas1HeatDir}, control law \eqref{eq:WContDefHeatDir} and $w(\cdot,0)\in \mathcal{D}(\mathcal{A}_1^{\frac{1}{2}})$. Let $\delta>0$ be a desired decay rate, $N_0\in \mathbb{N}$ satisfy \eqref{eq:N0HeatDir} and $N\in \mathbb{N}$ satisfy $N_0\leq N$. Let $L_0$ and $K_0$ be obtained using \eqref{eq:GainsDesignLHeatDir}  and \eqref{eq:GainsDesignKHeatDir}, respectively. Let there exist a positive definite matrix $P\in \mathbb{R}^{(2N+1)\times (2N+1)}$ and scalar $\alpha_1>0$ which satisfy \eqref{eq:WStabResultNonDelayedHeatDir2}. Then the solution $w(x,t)$ and $u(t)$ to \eqref{eq:PDE1HeatDir} under the control law \eqref{eq:WContDefHeatDir}, \eqref{eq:WobsODENonDelayedHeatDir1} and the corresponding observer $\hat{w}(x,t)$ defined by \eqref{eq:WhatSeriesHeatDir} satisfy
\begin{equation}\label{eq:WH1StabilityHeatDir}
\begin{array}{ll}
&\left\|w(\cdot,t)\right\|^2_{H^1}+ \left|u(t) \right|^2\leq Me^{-2\delta t}\left\|w(\cdot,0)\right\|^2_{H^1},\\
&\left\|w(\cdot,t)-\hat{w}(\cdot,t)\right\|^2_{H^1}\leq Me^{-2\delta t}\left\|w(\cdot,0)\right\|^2_{H^1},
\end{array}
\end{equation}
for some constant $M>0$. Moreover, \eqref{eq:WStabResultNonDelayedHeatDir2} is always feasible for large enough $N$.
\end{theorem}
\begin{proof}
Feasibility of the LMI \eqref{eq:WStabResultNonDelayedHeatDir2} implies, by the comparison principle,
\begin{equation}\label{eq:ComparisonDirichletHeatDir}
V(t)\leq e^{-2\delta t}V(0), \ t\geq 0.
\end{equation}
Since $u(0)=0$, for some $M_0>0$ we have
\begin{equation}\label{eq:VZeroNonDelayedHeatDir1}
\begin{array}{l}
V(0) \overset{\eqref{lem22}}{\leq} M_0\left\|w_x(\cdot,0) \right\|^2\leq M_0\left\|w(\cdot,0)\right\|^2_{H^1}.
\end{array}
\end{equation}
By Wirtinger's inequality (\cite{Fridman14_TDS}, Section 3.10),
for $t\geq 0$,
\begin{equation}\label{eq:WirtingerHeatDir}
\left\|w_x(\cdot,t) \right\|^2\leq \left\|w(\cdot,t) \right\|^2_{H^1}\leq \pi^{-2}\left\|w_x(\cdot,t) \right\|^2.
\end{equation}	
Since $w(\cdot,t)\in \mathcal{D}(\mathcal{A}_1)$ for all $t> 0$ we have $\left\|w_x(\cdot,t) \right\|^2\overset{\eqref{lem22}}{=}\sum_{n=1}^{\infty} \lambda_nw_n^2(t)$.
Parseval's equality, \eqref{eq:WirtingerHeatDir} and monotonicity of $\left\{\lambda_n \right\}_{n=1}^{\infty}$ imply
\begin{equation}\label{eq:VBoundBelowHeatDirichlet}
\begin{array}{ll}
&V(t)\geq \sigma_{min}(P)\left|u(t) \right|^2\\
&+ \operatorname{min}\left( \frac{\sigma_{min}(P)\pi^2}{2\lambda_N},\pi^2 \right)\left\|w(\cdot,t) \right\|^2_{H^1}, \ \ t\geq 0.
\end{array}
\end{equation}
Then \eqref{eq:WH1StabilityHeatDir} follows from \eqref{eq:ComparisonDirichletHeatDir}, \eqref{eq:VZeroNonDelayedHeatDir1}, \eqref{eq:VBoundBelowHeatDirichlet} and the presentation
\begin{equation*}
w(\cdot,t)-\hat{w}(\cdot,t)= \sum_{n=1}^Ne_n(t)\phi_n(\cdot) + \sum_{n=N+1}^{\infty}w_n(t)\phi_n(\cdot).
\end{equation*}
For feasibility of \eqref{eq:WStabResultNonDelayedHeatDir2} with large enough $N$, note that \eqref{eq:AssbnNonDelayed1HeatDir} and \eqref{eq:AsscnNonDelayed} imply $\left|c_n \right|\leq\sqrt{2}, \ n\geq 1$ and $\left\{b_n\right\}_{n=1}^{\infty}\in l^2(\mathbb{N})$. Then, by arguments of Theorem 3.2 in \cite{RamiContructiveFiniteDim}, there exist some $\Lambda,\kappa>0$, independent of $N$, such that
\begin{equation}\label{eq:FEstimateHeatDir}
\left|e^{\left(F+\delta I\right)t} \right|\leq \Lambda \cdot \sqrt{N}\left(1+t+t^2\right)e^{-\kappa t}.
\end{equation}
Therefore, $P\in \mathbb{R}^{(2N+1)\times(2N+1)}$ which solves the Lyapunov equation
\begin{equation}\label{eq:LyapEqHeatDir}
P(F+\delta I)+(F+\delta I)^T=-N^{-\frac{3}{4}}I
\end{equation}
satisfies
\begin{equation}\label{eq:PnormHeatDir}
\left|P\right|\leq \Lambda_1\cdot N^{\frac{1}{4}}
\end{equation}
where $\Lambda_1>0$ is independet of $N$. We substitute \eqref{eq:LyapEqHeatDir}, $\lambda_{N+1}=\pi^2(N+1)^2$ and $\alpha = N^{-\frac{3}{8}}$ into \eqref{eq:WStabResultNonDelayedHeatDir1}. By Schur complement, we find that \eqref{eq:WStabResultNonDelayedHeatDir1} holds if and only if
\begin{equation}\label{eq:EquivLMIsHeatDir}
\begin{array}{lll}
&\hspace{-3mm}-I+ 4\pi^{-\frac{3}{4}}N^{-
\frac{1}{8}}\tilde{K}_0^T\tilde{K}_0\\
&\hspace{-3mm}+\frac{1}{2}\left( \lambda_{N+1}-a-\delta-N^{\frac{3}{8}}\pi^{\frac{3}{2}}(N+1)^{\frac{3}{2}}\right)^{-1}P\mathcal{L}\mathcal{L}^TP<0.
\end{array}
\end{equation}
Since $\lambda_{N+1}-a-\delta\approx (N+1)^2$ and $\left|\tilde{K}_0\right|$, $\left|\mathcal{L}\right|$ are independent of $N$, by taking into account \eqref{eq:PnormHeatDir} we find that \eqref{eq:EquivLMIsHeatDir} holds for large enough $N$.
\end{proof}
\begin{corollary}\label{cor:H1ConvHeatDir}
	Under the conditions of Theorem \ref{Thm:WdynExtensionHeatDir}, the following estimates hold for $z(x,t)$ given in \eqref{eq:ChangeVarsHeatDir}:
	\begin{equation}\label{eq:ZH1StabilityHeatDir}
	\begin{array}{ll}
	&\left\|z(\cdot,t)\right\|^2_{H^1}\leq Me^{-2\delta t}\left\|z(\cdot,0)\right\|^2_{H^1},\\
	&\left\|z(\cdot,t)-\hat{w}(\cdot,t)\right\|^2_{H^1}\leq Me^{-2\delta t}\left\|z(\cdot,0)\right\|^2_{H^1},
	\end{array}
	\end{equation}
	where $M>0$ is some constant.
\end{corollary}
\begin{proof}
	From \eqref{eq:ChangeVarsHeatDir} we have
	\begin{equation}\label{eq:ZEstimateH1HeatDir}
	\begin{array}{lll}
	&\left\|z(\cdot,t) \right\|_{H^1}\leq \left\|w(\cdot,t) \right\|_{H^1}+\left|u(t)\right|\left\|r(\cdot) \right\|_{H^1},\\
	&\left\|z(\cdot,t) -\hat{w}(\cdot,t)\right\|_{H^1}\leq \left\|w(\cdot,t) -\hat{w}(\cdot,t)\right\|_{H^1}\\
	&\hspace{30mm}+\left|u(t)\right|\left\|r(\cdot) \right\|_{H^1}.
	\end{array}
	\end{equation}
	From $u(0)=0$, \eqref{eq:WH1StabilityHeatDir} and \eqref{eq:ZEstimateH1HeatDir},  we obtain \eqref{eq:ZH1StabilityHeatDir}.
\end{proof}
\begin{remark}
Differently from \cite{katz2020constructiveDelay}, where Dirichlet actuation with non-local measurements were considered, we apply the
Cauchy-Schwarz inequality in \eqref{eq:WCrosTermNonDelayedHeatDir} with fractional powers of $\lambda_n$ which allows to compensate $\zeta$
by using \eqref{eq:zetaestHeatDir} in the Lyapunov analysis. Note that for finite-dimensional observer-based control of the 1D linear Kuramoto-Sivashinsky equation (KSE), studied in \cite{Rami_CDC20}, the use of fractional powers of the eigenvalues was not required. This is due to the faster growth rate of the eigenvalues corresponding to the fourth order spatial differential operator appearing in the KSE.
\end{remark}

\section{Sampled-data control of  heat equation}
Consider now sampled-data control of the 1D linear heat equation \eqref{eq:LinearHeat} under Dirichlet actuation \eqref{eq:BCsHeatDir}. We introduce two sequences of sampling instances. For the first sequence, let $0=s_0<s_1<\dots<s_k<\dots$, $\lim_{k\to \infty}s_k=\infty$ be the measurement sampling instances.
We  consider discrete-time in-domain point measurement
\begin{equation}\label{eq:InDomPointMeasHeatDelayed}
y(t) = z(x_*,s_k), \ x_*\in (0,1),\ t\in [s_k, s_{k+1}).
\end{equation}
 We assume that $s_{k+1}-s_k\leq \tau_{M,y}$ for all $k=0,1,\dots$ and some constants $\tau_{M,y}>0$.

  For the second sequence, let $0=t_0<t_1<\dots<t_j<\dots$, $\lim_{j\to \infty}t_j=\infty$ be the controller hold times.
   We assume that $t_{j+1}-t_j\leq \tau_{M,u}$ for all $j=0,1,\dots$ and some constant $\tau_{M,u}>0$.
The control signal $u(t)$ is generated by a \emph{generalized hold} function 
\begin{equation}\label{eq:uSampAssump}
\dot{u}(t) =v(t_j), \quad t\in [t_j,t_{j+1})
\end{equation}
where $\left\{v(t_j)\right\}_{j=1}^{\infty}$ are to be determined.  Furthermore, we choose $u(0)=0$. By a generalized hold we mean the following: 
given $v(t_j)$, the value of the control signal is computed  as  (see Figure \ref{fig:Network})
\begin{equation}
\label{GH}
u(t)=u(t_j)+v(t_j)(t-t_j),\ \ t\in[t_j,t_{j+1}), \ \ j=0,1,2,...\end{equation}
%
The considered sampled-data control may correspond also to a networked control system with two independent networks (where network-induced delays are negligible): from sensor to controller with transmission instances $s_k$ and from controller to actuator with transmission instances $t_j$. In this case, $t_j$ are also the
updating times of the generalized hold device on the actuator side.
\begin{figure}
	\centering
	\includegraphics[width=55mm,scale=0.11]{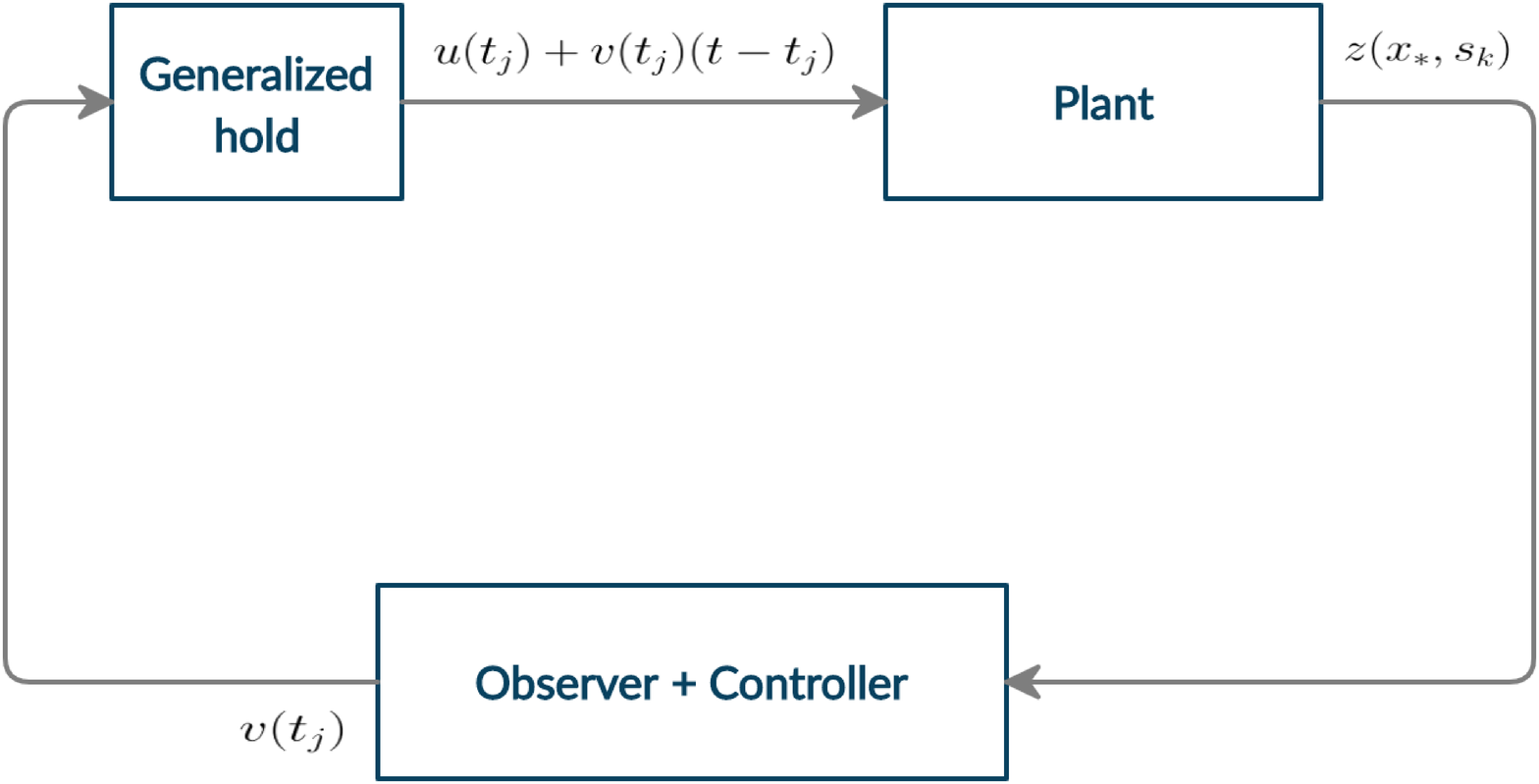}
	\caption{Network-based control with generalized hold.}\label{fig:Network}
\end{figure}

By the time-delay approach to sampled-data control (see \cite{Fridman14_TDS}), the measurement and input delays are presented as
\begin{equation}\label{eq:tau_yu}
\begin{array}{lll}
\tau_y(t) = t-s_k,\quad t\in [s_k,s_{k+1}),\\
\tau_u(t) = t-t_j,\quad t\in [t_j,t_{j+1}).
\end{array}
\end{equation}
Henceforth the dependence of $\tau_y(t), \tau_u(t)$ on $t$ will be suppressed to shorten the notations.

Introducing the change of variables \eqref{eq:ChangeVarsHeatDir} we obtain the following ODE-PDE system
\begin{equation}\label{eq:PDE1HeatDirSamp}
\begin{aligned}
& w_t(x,t)=w_{xx}(x,t)+aw(x,t)\\
&\hspace{12mm}+ar(x)u(t)-r(x)v(t-\tau_u),\\
& \dot{u}(t)=v(t-\tau_u), \quad t\geq 0.
\end{aligned}
\end{equation}
with boundary conditions \eqref{eq:PDE1HeatDirBCs} and measurement
\begin{equation}\label{eq:InDomPointMeas1HeatDirSamp}
y(t) = w(x_*,t-\tau_y)+r(x_*)u(t-\tau_y).
\end{equation}
Recall that we treat $u(t)$ as an additional state variable and $v(t-\tau_u)$ as the control input to be determined.

We present the solution to \eqref{eq:PDE1HeatDirSamp} as \eqref{eq:WseriesHeatDir} with $\left\{\phi_n\right\}_{n=1}^{\infty}$ defined in \eqref{eq:SLBCs}. By differentiating under the integral sign, integrating by parts and using \eqref{eq:SL} and \eqref{eq:2BCs} we obtain
\begin{equation}\label{eq:WOdesHeatDirSamp}
\begin{array}{lll}
&\hspace{-4mm}\dot{w}_n(t) = (-\lambda_n+a)w_n(t) + ab_nu(t) -b_nv(t-\tau_u), \ t\geq 0
\end{array}
\end{equation}
with $\left\{b_n\right\}_{n=1}^{\infty}$ given in \eqref{eq:WOdesHeatDir}. In particular, \eqref{eq:AssbnNonDelayedHeatDir} and \eqref{eq:AssbnNonDelayed1HeatDir} hold.

Given $\delta>0$, let $N_0 \in \mathbb{N}$ satisfy \eqref{eq:N0HeatDir} and $N\in \mathbb{N}, \ N_0\leq N$. $N_0$ will define the dimension of the controller, whereas $N$ will define the dimension of the observer.
We construct a finite-dimensional observer of the form \eqref{eq:WhatSeriesHeatDir} where $\hat{w}_n(t)$ satisfy the ODEs for $t\geq 0$
\begin{equation}\label{eq:WobsODENonDelayedHeatDir1Delay}
\begin{array}{lll}
&\dot{\hat{w}}_n(t) = (-\lambda_n+a)\hat{w}_n(t) +ab_nu(t)- b_nv(t-\tau_u)\\
&\hspace{8mm}-l_n\left[\hat{w}(x_*,t-\tau_y)+r(x_*)u(t-\tau_y)- y(t)\right],\\
&\hat{w}_n(0)=0, \ n\geq 1
\end{array}
\end{equation}
with $y(t)$ in \eqref{eq:InDomPointMeas1HeatDirSamp} and scalar observer gains $l_n,\ 1\leq n\leq N$.

Under Assumption 1 let the observer and controller gains, $L_0$ and $K_0$ , satisfy \eqref{eq:GainsDesignLHeatDir} and \eqref{eq:GainsDesignKHeatDir}, respectively. We choose $l_n=0$ for $N_0+1\leq n \leq N$.
We propose a $(N_0+1)$-dimensional controller of the form
\begin{equation}\label{eq:WContDefHeatDirDelay}
\begin{aligned}
& v(t-\tau_u)= K_0\hat{w}^{N_0}(t-\tau_u)
\end{aligned}
\end{equation}
with $\hat{w}^{N_0}(t)$ defined in \eqref{eq:WContDefHeatDir}. The proposed controller is based on the $N$-dimensional observer \eqref{eq:WhatSeriesHeatDir}.

Well-posedness of the closed-loop system \eqref{eq:PDE1HeatDirSamp} and \eqref{eq:WobsODENonDelayedHeatDir1Delay} with control input \eqref{eq:WContDefHeatDirDelay} follows from arguments of \cite{RamiContructiveFiniteDim}, together with the step method (i.e proving well-posedness step-by-step between consecutive sampling instances). Thus, the closed-loop system \eqref{eq:PDE1HeatDirSamp} and \eqref{eq:WobsODENonDelayedHeatDir1Delay} with control input \eqref{eq:WContDefHeatDirDelay} and initial condition $w(\cdot,0)\in \mathcal{D}\left(\mathcal{A}_1^{\frac{1}{2}}\right)$ has a unique solution
\begin{equation}\label{eq:Classical1Samp}
\begin{array}{lll}
&\xi \in C\left([0,\infty);\mathcal{H}\right)\cap C^1\left((0,\infty)\setminus \mathcal{J};\mathcal{H}\right), \\ &\mathcal{J}=\left\{t_j\right\}_{j=1}^{\infty}\cup\left\{s_k\right\}_{k=1}^{\infty}
\end{array}
\end{equation}
satisfying \eqref{eq:Classical2}.

Recall the estimation error $e_n(t)$ defined in \eqref{eq:WEstErrorNonDelayed}. By using \eqref{eq:WseriesHeatDir}, \eqref{eq:WhatSeriesHeatDir} and arguments similar to \eqref{eq:WIntroZetaNonDelayedHeatDir} the last term on the right-hand side of \eqref{eq:WobsODENonDelayedHeatDir1Delay} can be written as
\begin{equation}\label{eq:WIntroZetaNonDelayedHeatDirDelay}
\begin{array}{ll}
&\hat{w}(x_*,t)+r(x_*)u(t) - y(t) \\
&=-\sum_{n=1}^{N} c_ne_n(t-\tau_y)-\zeta(t-\tau_y)
\end{array}
\end{equation}
where $\zeta(t)$ is defined in \eqref{eq:zetaintegralHeatDir} and satisfies \eqref{eq:zetaestHeatDir}. Then the error equations have the form
\begin{equation}\label{eq:WenHeatDirDelay}
\begin{array}{ll}
&\dot e_n(t)=(-\lambda_n+a)e_n(t)-l_n\left(\sum_{n=1}^{N} c_ne_n(t-\tau_y)\right.\\
&\hspace{8mm}\left.+\zeta(t-\tau_y)\right),\quad t\geq 0.
\end{array}
\end{equation}
Recalling the notations \eqref{eq:ErrDefNonDelayedHeatDir} and \eqref{eq:ErrDefNonDelayedHeatDir1} we define
\begin{equation}\label{eq:ClosedloopMatDelay}
\begin{array}{lll}
& \Upsilon_y(t) = X(t-\tau_y)-X(t),\ \Upsilon_u(t) = X(t-\tau_u)-X(t),\\
&F_1 = \mathcal{L}\cdot [0,C_0,0,C_1]\in \mathbb{R}^{(2N+1)\times (2N+1)}, \\
&\mathcal{B}= \text{col}\left\{-\tilde{B}_0, 0,B_1,0 \right\}\in \mathbb{R}^{2N+1},\\
&\hat{K}_0 = \begin{bmatrix} K_0,&0\end{bmatrix} \in \mathbb{R}^{1\times 2N+1} .
\end{array}
\end{equation}
Then, using the notations \eqref{eq:ErrDefNonDelayedHeatDir}, \eqref{eq:ErrDefNonDelayedHeatDir1}  and  \eqref{eq:WOdesHeatDirSamp}, \eqref{eq:WobsODENonDelayedHeatDir1Delay}, \eqref{eq:WenHeatDirDelay}, \eqref{eq:ClosedloopMatDelay} we arrive at the following closed-loop system:
\begin{equation}\label{eq:ClosedLoopHeatDirDelay}
\begin{array}{lll}
&\dot{X}(t) = FX(t)+F_1\Upsilon_y(t)-\mathcal{B}\hat{K}_0\Upsilon_u(t)+\mathcal{L}\zeta(t-\tau_y),\\
& \dot{w}_n(t) = (-\lambda_n+a)w_n(t) -b_n\tilde{K}_0X(t)\\
&\hspace{10mm}-b_n\hat{K}_0\Upsilon_u(t), \ n>N, \quad t\geq 0.
\end{array}
\end{equation}
For $H^1$-stability analysis of the closed-loop system \eqref{eq:ClosedLoopHeatDirDelay} we fix  $\delta_0>\delta$ 
and define the Lyapunov functional
\begin{equation}\label{eq:VNonDelayedHeatDirDelay}
W(t)=V(t)+V_y(t)+V_u(t), \quad t\geq 0
\end{equation}
where $V(t)$ is defined in \eqref{eq:VNonDelayedHeatDir1} and
\begin{equation}\label{eq:DelayFunctionals}
\begin{array}{lll}
&V_y(t) = \tau_{M,y}^2\int_{t-\tau_y}^te^{-2\delta_0(t-s)}\dot{X}^T(s)W_1\dot{X}(s)ds\\
&- \frac{\pi^2}{4}\int_{t-\tau_y}^te^{-2\delta_0(t-s)}\Upsilon_y(s)^TW_1\Upsilon_y(s)ds, \ W_1>0,\vspace{0.1cm}\\
&V_u(t) = \tau_{M,u}^2\int_{t-\tau_u}^te^{-2\delta_0(t-s)}\dot{X}^T(s)\hat{K}_0^TW_2\hat{K}_0\dot{X}(s)ds\\
& - \frac{\pi^2}{4}\int_{t-\tau_u}^te^{-2\delta_0(t-s)}\Upsilon_u(s)^T\hat{K}_0^TW_2\hat{K}_0\Upsilon_u(s)ds, \ W_2>0.
\end{array}
\end{equation}
Note that $V_y(t),V_u(t)\geq 0$ due to the exponential Wirtinger inequality (see e.g \cite{selivanov2016observer}). We will employ the following Halanay's inequality for piecewise-continuous Lyapunov functions that do not grow at points of jump-discontinuity:
\begin{lemma}[Halanay's inequality]\label{prop:HalInterval}
	Let $s_0<...<s_k<...$ satisfy $\lim_{k\to\infty}s_k=\infty$ and $s_{k+1}-s_k\le h,\ k=0,1,\dots$.
	For any $k=0,1,\dots$, let there exist $\left\{t^{(k)}_j\right\}_{0\leq j\leq n_k}$ satisfying
	\begin{equation}\label{eq:RefinementSamp}
	s_k=t^{(k)}_0<t^{(k)}_1<\dots<t^{(k)}_{n_k-1}<t^{(k)}_{n_k}=s_{k+1}.
	\end{equation}
	Let $W:[s_0,+\infty)\rightarrow
	\R_+$ be absolutely continuous on $[t^{(k)}_j,t^{(k)}_{j+1})$ for all $k=0,1,\dots$ and $ 0\leq j\leq n_k$. Assume further that $W(t)$ satisfies
	\begin{equation}\label{eq:NojumpSamp2}
	\lim_{t\uparrow t^{(k)}_j}W(t)\geq W\left(t^{(k)}_j\right), \quad k=0,1,\dots, \ 0\leq j\leq n_k.
	\end{equation}
	and for $\delta_0>\delta_1>0$ and all $k=0,1,2,\dots$
	\begin{equation}
	\begin{array}{lll}\label{HalanayRAESamp1}
	\dot{W}(t) &\leq -2\delta_0{W}(t)+2\delta_1\sup_{s_k\leq{\theta}\leq
		t}W(\theta)\vspace{0,1cm}\\
	&{\mbox almost \ for \ all} \ t\in [t^{(k)}_j,t^{(k)}_{j+1}), \  0\leq j\leq n_k-1.
	\end{array}
	\end{equation}
	Then
	\begin{equation}
	\begin{array}{rr}\label{HalanayISSkDPlus}
	W(t) \le e^{-2\delta_{\tau} (t-s_0)} W(s_0), \quad t\geq s_0.
	\end{array}
	\end{equation}
	where  $\delta_{\tau}>0$ is a unique solution of
	\begin{equation}\label{eq:deltatau}
	\delta_{\tau}=\delta_0-\delta_1e^{2\delta_{\tau} h}.
	\end{equation}
\end{lemma}

The proof of Lemma is omitted due to the length limitation.
As the classical Halanay inequality (see \cite{Fridman14_TDS}), the proof of Lemma \ref{prop:HalInterval} is based on a comparison principle, where \eqref{HalanayRAESamp1} is taken into account.

Consider $[s_k,s_{k+1}), \ k=0,1,\dots$, where  $s_k,s_{k+1}$ are consecutive measurement sampling instances. Since the controller update instances satisfy $\lim_{j \to \infty}t_j = \infty$, there exist \emph{finitely many} controller update instances $t^{(k)}_j, \ 0\leq j\leq n_k$ for which \eqref{eq:RefinementSamp} holds. Furthremore, it can be easily verified that $W(t)$ defined by \eqref{eq:VNonDelayedHeatDirDelay},\eqref{eq:DelayFunctionals} is continuously differentiable on $[t^{(k)}_j,t^{(k)}_{j+1}), \ 0\leq j\leq n_k-1$ and satisfies \eqref{eq:NojumpSamp2}. Our goal now is to derive conditions which guarantee that \eqref{HalanayRAESamp1} holds.
Differentiating $V(t)$ on $[t^{(k)}_j,t^{(k)}_{j+1}), \ 0\leq j\leq n_k-1$ along the solution to \eqref{eq:ClosedLoopHeatDirDelay} we obtain
\begin{equation}\label{eq:WStabAnalysisNonDelayedHeatDirDelay}
\begin{array}{lll}
&\dot{V}+2\delta_0 V \leq X^T(t)\left[PF +F^TP+2\delta_0 P\right]X(t)\\
&+2X^T(t)PF_1\Upsilon_y(t)-2X^T(t)P\mathcal{B}\hat{K}_0\Upsilon_u(t)\\
&+2X^T(t)P\mathcal{L}\zeta(t-\tau_y)\\
&+2\sum_{n=N+1}^{\infty}\left(-\lambda_n^2+(a+\delta_0)\lambda_n\right)w_n^2(t)\\ &+2\sum_{n=N+1}^{\infty}\lambda_n w_n(t)b_n\left[-\tilde{K}_0X(t)-\hat{K}_0\Upsilon_u(t)\right].
\end{array}
\end{equation}
By arguments similar to \eqref{eq:WCrosTermNonDelayedHeatDir}
\begin{equation}\label{eq:WCrosTermNonDelayedHeatDirDelay}
\begin{array}{lllll}
&2\sum_{n=N+1}^{\infty}\lambda_n w_n(t)b_n\left[-\tilde{K}_0X(t)-\hat{K}_0\Upsilon_u(t)\right]\\
&\leq \left(\frac{1}{\alpha_1}+\frac{1}{\alpha_2}\right) \sum_{n=N+1}^{\infty}\lambda_n^{\frac{7}{4}} w_n^2(t)+ \frac{4\alpha_1 }{\sqrt{N}\pi^{\frac{3}{2}}}\left|\tilde{K}_0X(t)\right|^2\\
&+ \frac{4\alpha_2 }{\sqrt{N}\pi^{\frac{3}{2}}}\left|\hat{K}_0\Upsilon_u(t)\right|^2
\end{array}
\end{equation}
where $\alpha_i>0, \ i\in \left\{1,2\right\}$. Differentiating $V_y(t)$ and $V_u(t)$ along the solution to \eqref{eq:ClosedLoopHeatDirDelay} we obtain
\begin{equation}\label{eq:VyDot}
\begin{array}{lll}
&\dot{V}_y+2\delta_0V_y = \tau_{M,y}^2e^{2\delta_0\tau_{M,y}}\dot{X}^T(t)W_1\dot{X}(t)\\
&\hspace{17mm} -\frac{\pi^2}{4}\Upsilon_y(t)^TW_1\Upsilon_y(t),\\
&\dot{V}_u+2\delta_0V_u = \tau_{M,u}^2e^{2\delta_0\tau_{M,u}}\dot{X}^T(t)\hat{K}_0^TW_2\hat{K}_0\dot{X}(t)\\
&\hspace{17mm} -\frac{\pi^2}{4}\Upsilon_u(t)^T\hat{K}_0^TW_2\hat{K}_0\Upsilon_u(t).
\end{array}
\end{equation}

To compensate $\zeta(t-\tau_y)$ we use the following estimate:
\begin{equation}\label{eq:zetaintroducLMIsDelay}
\begin{array}{lll}
&\hspace{-4mm}-2\delta_1 \sup_{s_k\leq\theta \leq t}W(\theta)\leq -2\delta_1 V(s_k)\overset{\eqref{eq:tau_yu}}{\leq} -2\delta_1 V(t-\tau_y)\vspace{0.1cm}\\
&\hspace{-4mm}\overset{\eqref{eq:zetaestHeatDir}}{\leq} -2\delta_1X^T(t)PX(t)-2\delta_1\Upsilon_y^T(t)P\Upsilon_y(t)-2\delta_1\zeta^2(t-\tau_y)\vspace{0.1cm}\\
&\hspace{-4mm}-2\delta_1X^T(t)P\Upsilon_y(t)-2\delta_1\Upsilon_y^T(t)PX(t)
\end{array}
\end{equation}
where $\delta_0>\delta_1>0$. Let
\begin{equation*}
\begin{array}{lll}
&\eta(t)= \text{col}\left\{X(t),\zeta(t-\tau_y),\Upsilon_y(t),\hat{K}_0\Upsilon_u(t)\right\},\\
& R = [F, \mathcal{L},F_1,-\mathcal{B}].
\end{array}
\end{equation*}
From \eqref{eq:WStabAnalysisNonDelayedHeatDirDelay}, \eqref{eq:WCrosTermNonDelayedHeatDirDelay}, \eqref{eq:VyDot} and \eqref{eq:zetaintroducLMIsDelay} we have
\begin{equation}\label{eq:WStabResultNonDelayedHeatDirDelay}
\begin{array}{ll}
\mathcal{H}_W &= \dot{W}(t)+2\delta_0 W(t)-2\delta_1 \sup_{s_k\leq \theta \leq t}W(\theta)\\
& \leq \eta^T(t)\Psi^{(2)}\eta(t)+2\sum_{n=N+1}^{\infty}\mu_n\lambda_n w_n^2(t)\leq 0
\end{array}
\end{equation}
provided $\mu_{n} = -\lambda_n+\left[\sum_{i=1}^2 \frac{1}{2\alpha_i}\right]\lambda_n^{\frac{3}{4}}+a+\delta_0< 0$ for $n>N$ and
\begin{equation}\label{eq:PSi2def}
\begin{array}{lll}
\Psi^{(2)}&=\small\left[
\begin{array}{cc|cccc}
\Phi^{(1)}\ & P\mathcal{L}\ & P(F_1-2\delta_1I)\ & -P\mathcal{B} \\
* \ & -2\delta_1 \ & 0 \ & 0 \\ \hline
* \ & * \ & -\overline{W}_1\ & 0 \\
* \ & * \ & * \ & -\overline{W}_2
\end{array}
\right]\\
&+R^T\left(\varepsilon_y W_1+\varepsilon_u \hat{K}_0^TW_2\hat{K}_0\right)R<0
\end{array}
\end{equation}
where $\Phi^{(1)}$ is defined in \eqref{eq:WStabResultNonDelayedHeatDir1} and
\begin{equation*}
\begin{array}{lll}
&\overline{W}_1 = 2\delta_1P+\frac{\pi^2}{4}W_1, \overline{W}_2 = \frac{\pi^2}{4}W_2-\frac{4\alpha_2 }{\sqrt{N}\pi^{\frac{3}{2}}}, \\
& \varepsilon_y=\tau_{M,y}^2e^{2\delta_0 \tau_{M,y}},\ \ \varepsilon_u=\tau_{M,u}^2e^{2\delta_0 \tau_{M,u}}.
\end{array}
\end{equation*}
Furthermore, monotonicity of $\left\{\lambda_n\right\}_{n=1}^{\infty}$ and Schur complement imply that $\mu_n< 0$ for all $n>N$ if and only if
\begin{eqnarray}\label{eq:tailDelay}
\hspace{-2mm}\scriptsize\left[
\begin{array}{c|cc}
-\lambda_{N+1}+a+\delta_0 \ & 1 \ & 1 \\ \hline
* \ & -2\alpha_1\lambda_{N+1}^{-\frac{3}{4}} \ & 0 \\
* \ & * \ & -2\alpha_2\lambda_{N+1}^{-\frac{3}{4}}
\end{array}
\right]<0.
\end{eqnarray}
From \eqref{eq:WStabResultNonDelayedHeatDirDelay}, the LMIS \eqref{eq:PSi2def} and \eqref{eq:tailDelay} result in $\mathcal{H}_W\leq 0$ for $t\in [t^{(k)}_j,t^{(k)}_{j+1}), \ 0\leq j\leq n_k-1$. From \eqref{HalanayISSkDPlus} and \eqref{eq:deltatau} with $h=\tau_{M,y}$ we arrive at
\begin{equation}\label{eq:HalInterval}
W(t) \le e^{-2\delta_{\tau}t}W(0), \quad  t\geq 0,
\end{equation}
Summarizing, we have:
\begin{theorem}\label{Thm:WdynExtensionHeatDirDelay}
Consider \eqref{eq:PDE1HeatDirSamp} with boundary conditions \eqref{eq:PDE1HeatDirBCs}, in-domain point measurement \eqref{eq:InDomPointMeas1HeatDirSamp}, control law \eqref{eq:WContDefHeatDirDelay} and $w(\cdot,0)\in \mathcal{D}(\mathcal{A}_1^{\frac{1}{2}})$. Given $\delta>0$, let $N_0\in \mathbb{N}$ satisfy \eqref{eq:N0HeatDir} and $N\in \mathbb{N}$ satisfy $N_0\leq N$. Let $L_0$ and $K_0$ be obtained using \eqref{eq:GainsDesignLHeatDir}  and \eqref{eq:GainsDesignKHeatDir}, respectively. Given $\tau_{M,y},\tau_{M,u}>0$, $\delta_1>0$ and $\delta_0=\delta_1+\delta$, let there exist positive definite matrices $P,W_1\in \mathbb{R}^{(2N+1)\times (2N+1)}$ and scalars $\alpha_1,\alpha_2,W_2>0$ which satisfy \eqref{eq:PSi2def} and \eqref{eq:tailDelay}. Then the solution $w(x,t)$ and $u(t)$ to \eqref{eq:PDE1HeatDirSamp} under the control law \eqref{eq:WContDefHeatDirDelay}, \eqref{eq:WobsODENonDelayedHeatDir1Delay} and the corresponding observer $\hat{w}(x,t)$ defined by \eqref{eq:WhatSeriesHeatDir} satisfy \eqref{eq:WH1StabilityHeatDir} with $\delta$ replaced by $\delta_{\tau}$, given in \eqref{eq:deltatau}. The LMIS
\eqref{eq:PSi2def} and \eqref{eq:tailDelay} are always feasible for large enough $N$ and small enough $\tau_{M,y},\tau_{M,u}$.
\end{theorem}
\begin{proof}
The proof of \eqref{eq:WH1StabilityHeatDir} follows from arguments identical to Theorem \ref{Thm:WdynExtensionHeatDir}. The feasibility of \eqref{eq:PSi2def} and \eqref{eq:tailDelay} for large enough $N$ and small enough $\tau_{M,y},\tau_{M,u}$ follows from arguments similar to Theorem 3.1 in \cite{katz2020constructiveDelay}.
\end{proof}
\begin{corollary}\label{cor:H1ConvHeatDirDelay}
	Under the conditions of Theorem \ref{Thm:WdynExtensionHeatDirDelay} the estimates \eqref{eq:ZH1StabilityHeatDir} hold for $z(x,t)$ given in \eqref{eq:ChangeVarsHeatDir}.
\end{corollary}

\section{Numerical examples}
We demonstrate our approach to Dirichlet control of a 1D linear heat equation in two cases - non-delayed boundary control and sampled data boundary control. In both cases we choose $a=10$, which results in an unstable open-loop system. Furthermore, the gains $L_0$ and $K_0$ are found from \eqref{eq:GainsDesignLHeatDir} and \eqref{eq:GainsDesignKHeatDir}, respectively and are given by
\begin{equation}\label{eq:Gains}
L_0 = 0.7062, \quad K_0 = \begin{bmatrix}
-4.8237 & -5.2287
\end{bmatrix}.
\end{equation}
The LMIs of Thereoms \ref{Thm:WdynExtensionHeatDir} and \ref{Thm:WdynExtensionHeatDirDelay} were verified using the standard Matlab LMI toolbox. For non-delayed boundary control, we choose $\delta = 0.1$, which leads to $N_0=1$. The LMIs of Theorem \ref{Thm:WdynExtensionHeatDir} were found to be feasible for $N=4$. For sampled data boundary control we consider $\delta_{\tau}<\delta_0-\delta_1=0.1$, which leads to $N_0=1$. For $N\in \left\{6,8,10,12,14\right\}$ and $\delta_0=6$, the LMIs of Theorem \ref{Thm:WdynExtensionHeatDir} were verified in order to find the maximum values of $\tau_{M,y}$ and $\tau_{M,u}$ which result in feasibility. The results are presented in Table \ref{Tab:Sim}. It can be seen from Table \ref{Tab:Sim} that there is a trade-off between $\tau_{M,y}$ and $\tau_{M,u}$. In particular, increasing $\tau_{M,y}$ by $10^{-3}$ decreases $\tau_{M,u}$ by more than this amount. Furthermore, increasing $N$ preserves feasibility of the LMIs while increasing both $\tau_{M,y}$ and $\tau_{M,u}$. Numerical simulations of the closed-loop system for both cases confirm the theoretical results. The details are omitted due to length limitations.
\begin{table}[h]
	\begin{center}
		\footnotesize\begin{tabular}{|l|l|l|l|l|l|}
			\hline
			& N=6 & N=8 & N=10 & N=12& N=14  \\
			\hline
			$\tau_{M,y}$ &
			\multicolumn{5}{c|}{$\tau_{M,u}$}  \\
			\hline
			$0.002$ & 0.048 & 0.051 & 0.052 &0.053& 0.055\\
			\hline
			$0.004$ & 0.044 & 0.047 & 0.05 &0.051& 0.053\\
			\hline
			$0.006$ & 0.036 & 0.041 & 0.044 &0.047& 0.049\\
			\hline
			$0.008$ & 0.029 & 0.035 & 0.038 &0.041& 0.042\\
			\hline
			$0.01$ & 0.021 & 0.028 & 0.031 &0.034& 0.036\\
			\hline
			$0.012$ & 0.008 & 0.019 & 0.024 &0.027& 0.029\\
			\hline
			$0.014$ & - & 0.01 & 0.015 &0.018& 0.021\\
			\hline
			$0.016$ & - & - & 0.005 &0.009& 0.012\\
			\hline
		\end{tabular}
	\end{center}
	\caption{\label{Tab:Sim} Maximum value of $\tau_{M,u}$ for different values of $N$ and $\tau_{M,y}$.}
\end{table}
\section{Conclusions}
This paper presented the first constructive LMI-based method for finite-dimensional boundary controller design under the point in-domain measurement for 1D heat equation.
The method was based
on modal decomposition approach via dynamic extension.
Sampled-data implementation of the controller under sampled-data measurements was presented.
The proposed method can be extended to other PDEs and to input-to-state stabilization.


\bibliographystyle{IEEEtran}
\bibliography{IEEEabrv,Bibliography021218}

\end{document}